\documentclass[10pt]{amsart}
\usepackage{amsmath,amssymb,latexsym}
\usepackage[T1]{fontenc}
\usepackage{dsfont}
\usepackage{enumerate}

\usepackage{soul}

\usepackage{mdwlist}

\RequirePackage{ifpdf}
\ifpdf
\usepackage[pdftex]{hyperref}
\else
\usepackage[hypertex]{hyperref}
\fi

\def\M{\mathcal M}
\def\LL{\mathcal L}

\def\tp{\mathrm{tp}}
\def\D{\mathrm{Def}}
\def\FC{\mathrm{FC}}

\theoremstyle{plain}
\newtheorem{theorem}{Theorem}[section]
\newtheorem{theoremA}{Theorem}
\newtheorem{proposition}[theorem]{Proposition}
\newtheorem{fact}[theorem]{Fact}
\newtheorem{lemma}[theorem]{Lemma}
\newtheorem{cor}[theorem]{Corollary}

\theoremstyle{definition}
\newtheorem{definition}[theorem]{Definition}
\newtheorem{remark}[theorem]{Remark}

\newtheorem{example}[theorem]{Example}
\newtheorem*{akn}{Aknowledgments}

\title{Probabilistically-like nilpotent groups}
\author{Daniel Palac\'in}
\address{Departamento de \'Algebra, Geometr\'ia y Topolog\'ia, 	Facultad de Ciencias Matem\'aticas, 
	Universidad Complutense de Madrid, Plaza Ciencias 3, 28040, Madrid, Spain}
\email{dpalacin@ucm.es}
\thanks{Research supported by MTM2017-86777-P as well as by the Deutsche
	Forschungsgemeinschaft (DFG, German Research Foundation) - 
	Project number 2100310301, part of the ANR-DFG 
	program GeoMod}
\begin{document}

\begin{abstract}
The main goal of the paper is to present a general model theoretic framework to understand  a result of Shalev on probabilistically finite nilpotent groups. We prove that a suitable group where the equation $[x_1,\ldots,x_k]=1$ holds on a wide set, in a model theoretic sense, is an extension of a nilpotent group of class less than $k$ by a uniformly locally finite group. In particular, this result applies to amenable groups, as well as to suitable model-theoretic families of definable groups such as groups in simple theories and groups with finitely satisfiable generics. 
\end{abstract}

\maketitle

\section*{Introduction}

Given a finite group $G$, the finite direct product $G^{n} = G\times \stackrel{(n)}{\ldots} \times G$ is endowed with the structure of a probability space by considering the normalized counting measure. The probability that $n$ random elements of $G$ satisfy a non-trivial word $w=w(x_1,\ldots,x_n)$ is then defined as
\[
\mathrm{Pr}_{G,w}(1) = \frac{\left|w_{G}^{-1}(1)\right|}{|G|^n},
\]
where $w_{G} : G ^n\to G$ is the map induced by the word $w$ by substitution.

For the commutator word $u(x,y) =[x,y]$, Gustafson observed in \cite{Gu73} that $\mathrm{Pr}_{G,u}(1)$ equals the number of conjugacy of classes of $G$ divided by the order of $G$ and that $\mathrm{Pr}_{G,u}(1)\le 5/8$ for non-abelian groups. He also proved that with a natural interpretation, the latter inequality holds for compact Hausdorff non-abelian groups, by considering the product measure of the normalized Haar measure (cf.\,\cite[Section 2]{Gu73}). 
A theorem of Neumann \cite{pN89} asserts that for every constant $\epsilon>0$ there is some positive integer $m=m(\epsilon)$ such that given a finite group $G$ with $\mathrm{Pr}_{G,u}(1)\ge \epsilon$, there are two normal subgroups $N\le H$ of $G$ such that $H/N$ is abelian and the orders of $G/H$ and $N$ are bounded above by $m$. A similar result was obtained by L\'evai and Pyber \cite{LePy} for profinite groups, by showing that if the  set  of  commuting  pairs  of  a  profinite  group
has positive Haar measure then the group is abelian-by-finite. Within the same spirit other results have been obtained for infinite groups, such as amenable groups \cite{AMV,Toi}.

Left normed commutator words $w_k=w_k(x_1,\ldots,x_k)$ of longer length are defined inductively on $k$ as $w_1=x_1$ and $w_{k+1}=[w_{k},x_{k+1}]$. In \cite{Sha}, Shalev studied finite groups of a given probability $\mathrm{Pr}_{G,w_k}(1)$. He proved the following: for any $k\ge 2$ and any $\epsilon>0$, there exists some $r=r(k,\epsilon)$ such that given a finite group $G$ with $\mathrm{Pr}_{G,w_k}(1)\ge \epsilon$, there is a characteristic nilpotent subgroup $N$ of $G$ of class less than $k$ such that the group $G/N$ has exponent less than $r$. This result extends to arbitrary residually finite groups, by taking the probability in their profinite completion (cf.\,\cite[Theorem 1.1 \& Corollary 1.4]{Sha} as well as \cite[Theorem 1.19]{MTVV21}). 

In the aforementioned results, the group $G^{n}$ is regarded as a probability space and it is assumed that the set $w_{G,n}^{-1}(1) := (w_{n})_{G}^{-1} (1)$ is large with respect to the measure. From a model theoretic point of view, these results resonate to similar results obtained by Poizat \cite{bP} and Wagner \cite{fW91} in stable groups (cf.\,\cite{JW00,JW19}), where the notion of  having positive probability is replaced by being generic. 
Within this spirit, and following Hrushovski's approach in \cite{Hr0}, we shall present a general model theoretic framework to treat these phenomena in a uniform way. 

We shall consider groups $G$ that come equipped with a family of ideals on the boolean algebra of definable sets  which satisfy a compatibility Fubini-like property. For instance, ultraproducts of finite groups with the ideal of null-sets of the normalized counting measure is an archetypical example. Sets which do not belong to the ideal are called {\em wide}. We study the algebraic structure of these groups under the assumption that the set $w_{G,n}^{-1}(1)$ is wide and prove that such groups are extensions of nilpotent groups by uniformly locally finite groups (Corollary \ref{C:Main}). For finite groups, this yields the following result of Shalev \cite[Theorem 1.1 \& Corollary 1.4]{Sha}, which corresponds to Theorem \ref{T:Finite} in the sequel.

\begin{theoremA}\label{T:IntroFinite}
Given $\epsilon>0$ and natural numbers $d,k\ge 1$, there is some $r=r(\epsilon,d,k)$ such that given a finite group $G$ with $\mathrm{Pr}_{G,w_k}(1)\ge \epsilon$, there is some characteristic subgroup $N$ of $G$ which is nilpotent of class less than $k$ and such that every $d$-generator subgroup of $G/N$ is of order at most $r$.
\end{theoremA}
Our methods allow us to obtain new results concerning other families of groups. For example, a similar result is obtained for infinite amenable groups (Theorem \ref{T:Amen}), by choosing a suitable probability measure on the product space (cf.\,\cite[Theorem 1.14]{Toi} \& \cite[Theorem 1.6]{MTVV21}).
\begin{theoremA}\label{T:IntroAmen}
Let $G$ be an amenable group with a finitely additive right-invariant  probability measure $\mu$ such that 
\[
\int_{x_1\in G} \left( \dots \left( \int_{x_k\in G} \mathds{1}_{w_{G,k}^{-1}(1)} (x_1,\ldots,x_k) \,\mathrm{d} \mu \right) \ldots \right) \,\mathrm{d} \mu >0.
\]
We have that $G$ has a finite index subgroup which is FC-nilpotent of class at most $k-1$. Furthermore, there exists some characteristic subgroup $N$ of $G$ which is nilpotent of class less than $k$ and such that $G/N$ is uniformly locally finite.
\end{theoremA}

In particular, for finitely generated groups, this result yields that $G$ is an extension of a nilpotent group of class $k-1$ by a finite group, since finitely generated FC-nilpotent groups of class $n$ are (nilpotent of class $n$)-by-finite \cite[Corollary 2.1]{DM56}.

 The model theoretic approach taken here also applies to other families of groups which appear naturally in model theory, such as groups definable in simple theories as well as groups with finitely satisfiable generics (fsg, in short). Examples of groups with fsg are stable groups and definably compact groups definable in a saturated o-minimal expansion of a real closed field (see \cite[Theorem 8.1]{HPP}), such as compact Lie groups. In particular,  we obtain the following new result as well as its corresponding version for compact Lie groups (Corollary \ref{C:Lie}):

\begin{theoremA}
	Let $k\ge 1$ and let $G$ be a definable group with fsg such that $w_{G,k}^{-1}(1)$ is generic in $G^k$ (that is, finitely many translates of it cover $G^k$). We have that $G$ has a definable normal subgroup of finite index which is nilpotent of class at most $k-1$.
\end{theoremA}

\begin{akn}
 I wish to thank Amador Martin-Pizarro and Michael L\"osch for multiple discussions that helped to improve the presentation of the paper. I would also like to thank the anonymous referee for a very detailed report as well as for useful suggestions on improving the paper. In particular, on the definition of compatible system of ideals.
\end{akn}

\section{Some preliminaries on ideals}\label{s:Ideal}

 We work inside a structure $\M$ of a complete theory with infinite models in a countable language $\mathcal L$. We assume that the ambient model $\M$ is $\aleph_1$-saturated. By a definable set we mean a set definable with parameters from the universe of $\M$. Given a definable set $X$, we denote by $X(x)$ the first-order formula in $\mathcal L_x(\M)$ defining it, and vice versa. We write $\LL_X(\M)$ to denote the collection of formulas from $\LL_x(\M)$ that imply $X(x)$.
  
 We denote the Boolean algebra of definable subsets of a definable set $X$ by $\D_X(\M)$. An {\em ideal} $\mu$ of this boolean algebra is a non-empty collection a definable subsets of $X$ which does not contain $X$ and it is closed under subsets and finite unions. Sometimes we write $\mu_X$ to emphasize that it is an ideal in the Boolean algebra $\D_X(\M)$. Given an ideal $\mu$ of $\D_X(\M)$, we say that a type-definable set is {\em $\mu$-wide} if it it is not contained in a definable subset that belongs to $\mu$. We say that a partial type is $\mu$-wide if so is its type-definable set of realizations. By a standard compactness argument, wide partial types over a countable set of parameters can be extended to wide complete global types. Namely, given a wide partial type $\Psi(x)$ it suffices to see that the set 
\[ 
\Psi(x)\cup\{\neg \phi(x) : \phi(x)\in \LL_X(\M) \text{ non-wide} \}
\] 
is finitely consistent, since any completion of it yields a complete wide global type.

 We say that an ideal $\mu$ is {\em $A$-invariant} if it is invariant under type realizations, that is, if for every formula $\phi(x,y)$ in $\LL(A)$ and any two $|y|$-tuples $a$ and $b$ with $\tp(a/A)=\tp(b/A)$, the set $\phi(x,a)$ belongs to $\mu$ if and only if so does $\phi(x,b)$. We say that the ideal $\mu$ is {\em type-definable} (over a submodel $M$) if for every formula $\phi(x,y)\in\LL$ the set \[ \{b\in \M^{|y|} : \phi(x,b)\in \mu \} \] is type-definable over $M$. 
 
Before proceeding, let us recall some examples of ideals.

\begin{example}[Non-satisfiable sets]\label{E:FinSat}
Fix a subset $A\subset \M$, an $A$-definable set $X$ and consider the set of formulas $\phi(x)\in\mathcal L_X(\M)$ such that $\phi(\M)\cap A^{|x|}=\emptyset$. It is very easy to see that this is an $A$-invariant ideal.
\end{example}

 \begin{example}[Forking ideal]\label{E:Fork}
 
 A formula $\phi(x,a)$ {\em divides} over $A$ if there is an $A$-indiscernible sequence $(a_n)_{n\in\mathbb N}$ with $a_0=a$ such that the set $\{\phi(x,a_n)\}_{n\in\mathbb N}$ is inconsistent, and it {\em forks} over $A$ if it implies a finite disjunction of formulas such that each of them divides over $A$. Hence, it is straightforward to see that the set of formulas that fork over $A$ is an $A$-invariant ideal. Note that the forking ideal is proper whenever the formula $x=x$ does not fork.
\end{example}

\begin{example}[Keisler measure]\label{E:Meas}

A {\em Keisler measure} $\mu$ on $\mathcal M$ is a finitely additive probability measure on the boolean algebra $\D_X(A)$. Given such a measure, we consider the ideal of $\mu$-null sets. In general, this ideal need not be invariant nor type-definable. However, every finitely additive probability measure $\mu$ on all subsets of $X$ admits an expansion of the original language $\LL$ in which it becomes definable without parameters, see for example \cite[Section 2.6]{Hr1} (cf. Section 3.2). Namely, add a predicate $Q_{r,\varphi} (y)$ for each $r$ in $\mathbb Q \cap [0,1]$ and every formula  $\varphi(x,y)$ in $\LL$  
such that $Q_{r,\varphi}(b)$ holds if and only if $\mu(\varphi(X,b)) \le r$. These predicates $Q_{r,\varphi}$  give rise to new definable sets, which will also be measurable. Iterating this process countably many times and replacing the ambient model (if necessary), we obtain an expansion of the language $\LL$ such that the corresponding Keisler measure satisfies that the ideal of $\mu$-null sets becomes type-definable without parameters. In particular, a formula of positive measure does not fork over $\emptyset$,
 see \cite[Lemma 2.9 \& Example 2.12]{Hr1}. 
\end{example}

In the presence of a definable group, we should ask for some compatibility between the ideal and the group operation.

Let $G$ denote a definable group and let $\mu$ be an ideal on $\D_G(\M)$. We say that $\mu$ is {\em right translation-invariant} if the right translate of any set of $\mu$ also belongs to $\mu$. Likewise, we define the notion of left translation-invariant.
\begin{definition}
We say that $\mu$ has the right translation $\mathrm{S}1$ property if it is right translation-invariant and moreover, for every $A$-definable set $X$ there is no $A$-indiscernible sequence $(g_n)_{n\in\mathbb N}$ such that $Xg_n\cap Xg_m$ is in $\mu$ for $n<m$ but $Xg_n$ is $\mu$-wide for every (some) $n$.
\end{definition}

Note that if $G$ is an amenable group, that is a group equipped with a finitely additive right-invariant probability measure in $\mathcal P(G)$, then the ideal of subsets of measure $0$ has the right translation S$1$ property. In fact, the S$1$ property originates in the study of definable groups in pseudo-finite fields \cite{Hr0} and was extended in \cite{Hr1}, where it is defined with respect to automorphisms instead of translations (cf.\,\cite[Definition 2.8]{Hr1}). A similar argument as in \cite[Lemma 3.2]{Hr1} gives the following two results.

\begin{fact}\label{F:FinIndex}
Let $G$ be a definable group and let $\mu$ be an ideal on $\D_G(\M)$. Assume that $\mu$ is right translational $\mathrm{S}1$ and let $X$ be a definable subset. If $X$ is $\mu$-wide, then finitely many right translates of $X^{-1}X$ cover $G$. 
\end{fact}
\begin{proof}
Let $Y$ be a maximal subset of $G$ satisfying the property that for any two distinct elements $y_1,y_2$ of $Y$ we have that $Xy_1\cap Xy_2=\emptyset$. Since any ideal with the right translation $\mathrm{S}1$ property is right translation-invariant, the assumption yields that $Y$ is finite, as otherwise using compactness we could find an indiscernible sequence (over the parameters of $X(x)$) witnessing this. Thus, given $g$ in $G$ we can find some element $y$ in $Y$ such that $Xg\cap Xy\neq \emptyset$ and so $g$ belongs to $X^{-1}Xy$. Hence, as $g$ was arbitrary, we conclude that $G$ is contained in $X^{-1}XY$.
\end{proof}

The following fact is \cite[Lemma 3.3]{Hr1}.

\begin{fact}
 Assume that $\M$ is $\kappa$-saturated for some $\kappa>2^{|\LL|+\aleph_0}$ and let $A$ be a countable set. Let $G$ be a definable group and let $\mu$ be an ideal on $\D_G(\M)$. Assume that $\mu$ is right translational $\mathrm{S}1$. Then, a type-definable over $A$ subgroup of $G$ is $\mu$-wide if and only if it has index at most $2^{|\LL|+\aleph_0}$.
\end{fact}
 Within the same spirit we prove a similar result for ind-definable groups. We remind here that by an \emph{ind-definable over $A$ subgroup} of a definable group we mean a subgroup whose domain has an $A$-type-definable complement. That is, an ind-definable over $A$ subgroup is a subgroup $\bigcup_{i\in I} X_i$ where each $X_i$ is $A$-definable.
 
\begin{lemma}\label{L:WideInd}
Let $G$ be a definable group and let $\mu$ be an ideal on $\D_G(\M)$. Suppose that $A$ is a countable set and assume that $\mu$ is right translational $\mathrm{S}1$. Then, an ind-definable over $A$  subgroup of $G$ contains an $A$-definable $\mu$-wide set if and only if it is $A$-definable and has finite index.
\end{lemma}
\begin{proof}
 Suppose that an ind-definable over $A$ subgroup $H$ of $G$ contains a definable set $X$ which is $\mu$-wide. By Fact \ref{F:FinIndex}, finitely many right translates of the subset $X^{-1}X$ of $H$ cover $G$ and so $H$ has finite index in $G$. As a translate of $H$ is also ind-definable over $A$ and a representative of the coset, we then have that $H$ is type-definable and hence $A$-definable, by a standard compactness argument, using that the ambient model $\M$ is $\aleph_1$-saturated.
 
 The other direction is clear: a definable subgroup of finite index is $\mu$-wide, since $\mu$ is right translation-invariant.
\end{proof}

We recall a Fubini-like property for ideals introduced in \cite[Lemma 2.17]{Hr1}.

Identify $m$ with the set $\{0,1,\ldots,m-1\}$. Given definable sets $X_0,\ldots,X_{m-1}$ and a subset $s$ of $m$, set $X_s$ to be the Cartesian product $\prod_{i\in s} X_i$, ordered increasingly. Assume that the set $X_s$ is defined by the formula $X_s(x_s)$ with variables $x_s=(x_i)_{i\in s}$. 
\begin{definition}\label{D:Fubini}
For each $s\in\mathcal P(m)\setminus\{\emptyset\}$, let $\mu_{s}$ be an ideal on $\D_{X_s}(\M)$ and let $A$ be a small subset. We say that the family $\{\mu_{s}\}_{s\in\mathcal P(m)\setminus\{\emptyset\}}$ of ideals satisfies Fubini if the following properties are satisfied:
\begin{enumerate}[$(i)$]
 \item For any non-empty disjoint $s,t$ in $\mathcal P(m)$ we have that if $\phi(x_s)\wedge \psi(x_t)$ is in $
 \mu_{s\cup t}$, then either $\phi(x_s)$ is in $\mu_s$ or $\psi(x_t)$ is in $\mu_{t}$. 
 \item For any non-empty disjoint $s,t\in \mathcal P(m)$ and any $\LL(A)$-formula $\phi(x_s,x_t)$ we have that if $\phi(x_s,a_t)\in \mu_s$ whenever $\tp(a_t/A)$ is $\mu_t$-wide, then $\phi(x_s,x_t)\in \mu_{s\cup t}$.
 \item For any non-empty disjoint $s,t\in \mathcal P(m)$ and any $\LL(A)$-formula $\phi(x_s,x_t)$ we have that if $\phi(a_s,x_t)\in \mu_t$ whenever $\tp(a_s/A)$ is $\mu_s$-wide, then $\phi(x_s,x_t)\in \mu_{s\cup t}$.
\end{enumerate}
Note that the points $(ii)$ and $(iii)$ are equivalent, as one may reorder the variables of a formula. When $X_s=X_t$ and  $\mu_s=\mu_{t}$ for every $s,t\in \mathcal P(m)$ with $|s|=|t|$, we simply say that $\{\mu_n\}_{1\le n<m}$ satisfies the Fubini property.
\end{definition}
The definition is an abstraction of the situation in finitely additive probability measurable spaces where one considers the product measure. However, note that, as in Hrushovski's orginal definition, in $(i)$ it is not asked that $\phi(x_s)\wedge \psi(x_t)$ is in $
\mu_{s\cup t}$ if and only if $\phi(x_s)$ is in $\mu_s$ or $\psi(x_t)$ is in $\mu_{t}$.

Due to the nature of the left normed commutator word $w_k(x_1,\ldots,x_k)$ we can weaken the above property for our purposes. Instead of assuming that $G$ has a family of ideals satisfying Fubini, it would be sufficient to suppose that the  ideals have certain compatibility as in $(ii)$ and $(iii)$. 


\begin{definition}
Let $X_1,\ldots,X_{n}$ be definable sets and set $X= X_1\times \ldots \times X_{n}$. For each $1\le i<n$, suppose that $\mu_i$ is an ideal on $\D_{X_{i}}(\M)$ and let $\mu$ be an ideal on $\D_X(\M)$. We say that $\mu$ is {\em coordinatewise $(\mu_1,\ldots,\mu_{n})$-compatible} if the following holds:
\begin{enumerate}[$(\dagger)$]
	\item For any parameter set $A$ and any $A$-definable set $Y\in \D_X(\M)$ which is $\mu$-wide, there exist some $(a_1,\ldots,a_{n})\in Y$ such that the type $\tp(a_1/A)$ is $\mu_1$-wide and also the types $\tp(a_{i+i}/A,a_1,\ldots,a_{i})$ for $1\le i<n$ are $\mu_{i+1}$-wide.
\end{enumerate}
\end{definition}
\noindent Note that if $\nu\supset \mu$ is an ideal and $\mu$ is coordinatewise  $(\mu_1,\ldots,\mu_{n})$-compatible, then so is $\nu$. Also, if $\mu$ is coordinatewise $(\mu_1,\ldots,\mu_{n})$-compatible and $Y$ is a $\mu$-wide definable subset of $X$, then the projection of $Y$ onto the $i$-th coordinate is $\mu_i$-wide. Nonetheless, there is no information on pairs of coordinates. This phenomena is covered in the following notion:

\begin{definition}\label{D:CSI}
Let $X$ be a definable set and let $m\ge 1$. For each $1\le n<m$ let $\mu_n$ be an ideal of definable subsets of $X^n = X\times \stackrel{n}{\ldots} \times X$. We say that the family $\{\mu_{n}\}_{1\le n<m}$ of ideals of $X^m$ is a {\em compatible $m$-system of ideals} if $\mu_n$ is coordinatewise $(\mu_i,\mu_{n-i})$-compatible for every $1\le i<n$.
\end{definition}

Note that in the definition of compatible system we are assuming that any Cartesian product of the set $X$ of length $k$ is equipped with the ideal $\mu_k$. A compatible $2$-system of ideals is the same as a coordinatewise system (of length $2$). Nonetheless, in general these notions may differ. Their relation is collected in the following statement.

\begin{lemma}\label{L:Relation} Let $m\ge 1$ be a natural number. For $1\le n<m$, let $\mu_n$ be an ideal on definable sets of $X^n$. 
	\begin{enumerate}
		\item  If $\{\mu_n\}_{1\le n<m}$ satisfies the Fubini property, then $\{\mu_n\}_{1\le n<m}$ is a compatible $m$-system of ideals. 
		\item If $\{\mu_n\}_{1\le n<m}$ is a compatible $m$-system of ideals, then each $\mu_n$ is coordinatewise $(\mu_1,\ldots,\mu_1)$-compatible. 
	\end{enumerate}
\end{lemma}
\begin{proof}
(1) Fix some $1\le i<n$ and let $\phi(x_i,x_{n-i})$ be an $\LL(A)$-formula defining a $\mu_n$-wide subset $Y$ of $X^n$. By $(iii)$ of Fubini property, there must be some $a\in X^i$ such that $\tp(a/A)$ is $\mu_i$-wide but $\phi(a,x_{n-i})$ is not in $\mu_{n-i}$. Thus, there is some $b\in X^{n-i}$ such that $(a,b)\in Y$ holds $\tp(b/A,a)$ is $\mu_{n-i}$-wide. This shows that $\mu_n$ is coordinatewise $(\mu_i,\mu_{n-i})$-compatible. 

(2) The proof is by induction on $n$, with the case $n=1$ being trivial. Suppose $n>1$ and let $Y$ be an $A$-definable subset of $X^n$ which we assume to be $\mu_n$-wide. Thus, there is some $a_1\in X$ and $b\in X^{n-1}$ such that $(a_1,b) \in Y$ with $\tp(a_1/A)$ $\mu_1$-wide and $\tp(b/A,a_1)$ $\mu_{n-1}$-wide. As the set $Y_{a_1}=\{ y\in X^{n-1} \ : \ (a_1,y)\in Y\}$ is $\mu_{n-1}$-wide since it is definable over $A\cup\{a_1\}$ and $b\in Y_{a_1}$, by induction on $n$ we get some $(a_2,\ldots,a_n)\in Y_{a_1}$ such that $\tp(a_{i+1}/A,a_1,\ldots,a_{i})$ is $\mu_1$-wide for $1\le i<n$. This shows that the ideal $\mu_n$ is  coordinatewise $(\mu_1,\ldots,\mu_1)$-compatible. 
\end{proof}

Note that while Fubini is a symmetric property, compatibility is not. This is clear in the proof of (1) where we have only used condition (iii) of Fubini. In a way, Fubini is a commutativity-like property whereas compatibility is a non-symmetrical associativity-like property.

In Section \ref{s:Applications} we will provide examples of groups admitting a natural compatible system of ideals. In fact, note that Fubini's theorem for finite sums yields that a non-principal ultraproduct $G$ of finite groups, which is an $\aleph_1$-saturated group, admits a family of ideals $\{\mu_n\}_{n\in \mathbb N}$ satisfying Fubini, where each $\mu_n$ denotes the ideal of definable null sets with respect to the non-standard normalized counting measure on $G^{n}$ (cf. Section \ref{s:FiniteGroups} \& \cite[Theorem 19]{BT14}).

\section{Largely nilpotent groups}\label{s:Groups}
As in the previous section, we work inside an $\aleph_1$-saturated structure $\M$.

Given a group $G$ and a subset $X$ of $G$, we denote by $C_G(X)$ the set of elements of $G$ that commute with all elements of $X$. When $X$ is a definable set, the subgroup $C_G(X)$ is definable over the same parameters. Given a normal subgroup $N$ of $G$, we define $C_G(X/N)$ to be the subgroup 
\[
C_G(X/N) = \bigcap_{x\in X} C_G(x/N),
\]
where we write $C_G(x/N)$ to denote the preimage of $C_{G/N}(xN)$ in $G$ under the quotient projection. 

\begin{remark}\label{R:Ind-def}
	If $N$ is ind-definable and $X$ is type-definable, then $C_G(X/N)$ is ind-definable over the same parameters: indeed, an element $u$ does not belong to $C_G(X/N)$ if and only if it satisfies the partial type on $u$ given by
	\[
	\exists x y \left(X(x) \wedge (\neg N)(y) \wedge y = [u,x] \right).
	\]
\end{remark}

\begin{definition}
Let $G$ be a group and let $H,N$ be subgroups of $G$ with $N$ normal in $G$. The {\em FC-centralizer of $H$ modulo $N$ in $G$} is defined as
$$
\FC_G(H/N)=\{x\in G: |H : C_H(x/N)| \mbox{ is finite}\}.
$$
If $N$ is trivial it is omitted.
\end{definition}
\begin{remark}\label{R:FC-ind-def}
The FC-centralizer $\FC_G(H/N)$ is a subgroup, as $C_H(g/N)\cap C_H(h/N)$ is contained in $C_H(gh^{-1}/N)$ for any $g,h\in G$. Furthermore, if $N$ is an $\emptyset$-ind-definable normal subgroup of $G$, then the subgroup $\mathrm{FC}_G(G/N)$ is $\emptyset$-ind-definable as well: indeed, the complement of  $\mathrm{FC}_G(G/N)$ is type-defined by the partial type on $x$ given by:
\[
\bigwedge_{k\in\mathbb N}\exists (y_i)_{i\le k} \,\exists (u_{i,j})_{i<j\le k} \bigwedge_{i<j}\left([x,y_i^{-1}y_j]= u_{i,j} \wedge (\neg N)(u_{i,j}) \right),
\]
which by compactness expresses that there are infinitely many elements $(y_i)_{i\in\mathbb N}$ of $G$ such that $y_i^{-1}y_j$ does not belong to $C_G(x/N)$ for every $i<j$.
\end{remark}
The {\em FC-center} $\mathrm{FC}(G)$ of a group $G$ is simply $\FC_G(G)$. The terminology FC stands for finite conjugates, as in fact we have 
\[
\mathrm{FC}(G) = \{ g\in G \ : \ g^G \text{ is finite}\}.
\] This group might not be definable, since the finite indexes of the centralizers can be arbitrarily large, but it is clearly ind-definable in the pure language of groups: indeed, we have that
$$
\FC(G) = \bigcup_{n\in\mathbb N} \left\{ g\in G : |G:C_G(g)| \le n \right\}.
$$
When the increasing chain of definable subsets of the right hand side stabilizes, the FC-center is not far from being abelian. Neumann \cite[Theorem 3.1]{bN54} proved the following.

\begin{fact}\label{F:FC}
If there exists some $k$ such that $ \mathrm{FC}(G) = \{g\in G : |G:C_G(g)|\le k \}$, then $\mathrm{FC}(G)$ is finite-by-abelian.
\end{fact}

\begin{lemma}\label{L:key}
Let $G$ be a definable group and let $n\ge 1$. Suppose that $G$ has an ideal $\mu$ with the right translational $\mathrm{S}1$ property and let $N$ be an $\emptyset$-ind-definable normal subgroup of $G$. For any two elements $a,b \in G$ such that $\tp(a/b)$ is $\mu$-wide and $[b,a]\in N$, we have that $b\in \mathrm{FC}_G(G/N)$. Furthermore, if $\tp(b)$ is $\mu$-wide, then $G/N$ is finite-by-abelian-by-finite.
\end{lemma}
\noindent In particular, given a non-trivial word  $w=w(x_1,\ldots, x_{n})$ and elements $a_1,\ldots,a_{n+1}$ of $G$ such that $\tp(a_{n+1}/a_1,\ldots,a_{n})$ is $\mu$-wide and $[w(a_1,\ldots,a_{n}),a_{n+1}]\in N$, we have that 
$w(a_1,\dots,a_{n}) \in \mathrm{FC}_G(G/N).
$
\begin{proof}
Observe first that for any element $g$ of $G$, the subgroup $C_G(g/N)$ is ind-definable over $g$ by Remark \ref{R:Ind-def}. Assume that $a$ and $b$ are elements of $G$ such that $\tp(a/b)$ is $\mu$-wide and that $[b,a]$ belongs to $N$. As $a$ belongs to the subgroup $C_G(b/N)$, which is ind-definable over $b$ by Remark \ref{R:Ind-def}, we have that $C_G(b/N)$ contains a $\mu$-wide subset and so it has finite index in $G$ by Lemma \ref{L:WideInd}. Consequently, the element $b$ belongs to $\mathrm{FC}_G(G/N)$, as desired.

Suppose in addition that $b$ has a $\mu$-wide type. Since $\mathrm{FC}_G(G/N)$ is $\emptyset$-ind-definable by Remark \ref{R:FC-ind-def}, it contains a $\mu$-wide set. Thus, it is definable and has finite index in $G$ by Lemma \ref{L:WideInd}. Hence, by compactness and Remark \ref{R:FC-ind-def}, there is some natural number $k$ such that for every element $x$ in $G$, we have that $x$ belongs to $\mathrm{FC}_G(G/N)$ if and only if $|G:C_G(x/N)|\le k$. In particular, we have that
\[
\mathrm{FC}(G/N) = \{ xN\in G/N : |G/N : C_{G/N}(xN)|\le k\}.
\]Therefore, we have that $\mathrm{FC}(G/N)$ is finite-by-abelian by Fact \ref{F:FC} and hence $G/N$ is finite-by-abelian-by-finite. 
\end{proof}

As a consequence we obtain a generalization of Neumann's theorem \cite{pN89}.
\begin{cor}\label{C:1}
Let $G$ be a definable group. Suppose that $G$ admits a compatible $2$-system of ideals of $G\times G$ with $\mu_G$ satisfying the right translational $\mathrm{S}1$ property. Assume further that the set 
$$
w_{G,2}^{-1}(1) = \{(x,y)\in G^2 : [x,y]=1 \}
$$ 
is $\mu_{G^{2}}$-wide. We have that the group $G$ is finite-by-abelian-by-finite.
\end{cor}
\begin{proof}
Since the set of elements $(x,y)$ of $G\times G$ such that $x$ and $y$ commute is definable without parameters and $\mu_{G^{2}}$-wide, there exists a commuting pair $(a,b)$ such that the types $\tp(a)$ and $\tp(b/a)$ are $\mu_G$-wide by the assumption on the compatibility of ideals. Thus, by Lemma \ref{L:key} we obtain the statement.
\end{proof}
Before proceeding, we introduce some notation and recall the following generalization of nilpotent group, as well as of FC-group.
\noindent For a natural number $n$, the {\em $n$th iterated FC-center of $G$} is defined inductively by
$$
\FC_0(G)=\{1\} \ \mbox{ and } \ \FC_{n+1}(G)=\FC_G(G/\FC_n(G)).
$$
Note by Remark \ref{R:FC-ind-def} that the subgroups $\mathrm{FC}_n(G)$ are $\emptyset$-ind-definable for every $n$.

\begin{definition}
	A group $G$ is {\em FC-nilpotent} of class $n$ if $G=\mathrm{FC}_n(G)$. 
\end{definition}

\begin{proposition}\label{P:M1}
 Let $G$ be a definable group and let $n\ge 2$. Suppose that $G$ admits a coordinatewise $(\mu_G,\ldots,\mu_G)$-compatible system of ideals of $G^n$ with $\mu_G$ having the right translational $\mathrm{S}1$ property. Assume further that the set 
 $$
 w_{G,n}^{-1}(1) = \left\{(x_1,\ldots,x_{n})\in G^n : [x_1,\ldots,x_{n}] = 1 \right\}
 $$
 is $\mu_{G^{n}}$-wide. We have that $\mathrm{FC}_{n-1}(G)$ has finite index in $G$. In particular, the group $G$ has an $\emptyset$-definable finite index characteristic subgroup which is FC-nilpotent of class less than $n$.
\end{proposition}
\noindent Note that as a consequence the group $G$ is FC-nilpontent of class $n$.
\begin{proof}
Assume $n\ge 2$ and let $(a_1,\ldots,a_{n})$ be an $n$-tuple of elements of $G$ such that $[a_1,\ldots,a_{n}]=1$ and that the types $\tp(a_1)$ and $\tp(a_{k+1}/a_1,\ldots,a_{k})$ are $\mu_G$-wide for $k=1,\ldots,n-1$.  We show recursively on $i\le n-1$ that 
$$
[a_1,\ldots,a_{n-i}] \in \mathrm{FC}_{i}(G).
$$
Since the subgroup $\mathrm{FC}_0(G)$ is the trivial subgroup, the claim holds trivially for $i=0$. Now, suppose that we have already proved that
$$
[a_1,\ldots,a_{n-i}] \in \mathrm{FC}_{i}(G)
$$
holds. Since $\tp(a_{n-i}/a_1,\ldots,a_{n-(i+1)})$ is $\mu_G$-wide, using Lemma \ref{L:key} we obtain that  
$$
[a_1,\dots,a_{n-(i+1)}] \in \mathrm{FC}_G(G/\mathrm{FC}_i(G)) = \mathrm{FC}_{i+1}(G),
$$ 
as desired. Hence, this yields that $a_1$ belongs to $\mathrm{FC}_{n-1}(G)$ and so we obtain by Lemma \ref{L:WideInd} that $\mathrm{FC}_{n-1}(G)$ is definable and has finite index in $G$, since $\tp(a_1)$ is $\mu_G$-wide. 

Finally, set $H = \mathrm{FC}_{n-1}(G)$. Since $H$ has  finite index in $G$, we see inductively on $i<n-1$ that 
\[
H\cap \mathrm{FC}_{i+1}(G)=\mathrm{FC}_H(G/\mathrm{FC}_i(G)) = \mathrm{FC}_H(H/H\cap \mathrm{FC}_i(G)) = \mathrm{FC}_{i+1}(H)
\] and so $H=\mathrm{FC}_{n-1}(H)$, as desired.  
\end{proof}

\begin{theorem}\label{T:M2}
 Let $G$ be a definable group. Suppose that $G$ admits a compatible system of ideals $\{\mu_{G^k}\}_{1\le k\le n}$ and that $\mu_G$ has the right translational $\mathrm{S}1$ property. Let $N$ be a definable normal subgroup of $G$  and assume that the set $w_{G,n}^{-1}(N)$
 is $\mu_{G^{n}}$-wide. We have that there exists a definable, over the same parameters as $N$, normal subgroup $H$ of $G$ such that $H/N$ is nilpotent of class less than $n$ and $G/H$ has finite exponent.
\end{theorem}
\begin{proof}We may assume that $N$ is definable without parameters. We prove the statement by induction on $n$. For $n=1$, note that our assumptions yield that the definable subgroup $N$ is $\mu_G$-wide and thus  $G/N$ is finite by Fact \ref{F:FinIndex}. Hence, it suffices to set $H=N$.
	
Assume $n\ge 2$. As $\{\mu_{G^k}\}_{1\le k\le n}$ is a compatible system of ideals, we can find $a_1,\ldots,a_{n}$ of $G$ such that $\tp(a_1,\ldots,a_{n-1})$ is $\mu_{G^{n-1}}$-wide and $\tp(a_{n}/a_1,\ldots,a_{n-1})$ is $\mu_{G}$-wide, and that $[a_1,\ldots,a_n]\in N$. Thus $a_n\in C_G([a_1,\ldots,a_{n-1}]/N)$. Since $C_G([a_1,\ldots,a_{n-1}]/N)$ is definable over $a_1,\ldots,a_{n-1}$, it is $\mu_G$-wide and so has finite index in $G$, say $k$, by Fact \ref{F:FinIndex}. Let $X$ be the set 
$$
X=\left\{ x\in G : | G: C_G(x/N) | \le k \right\}
$$ 
and note that it is definable without parameters, as so is $N$. Set $H$ to be the definable normal subgroup $C_G(X/N)$ and let $K=C_G(H/N)$, a definable normal group. Note that both groups contain $N$ and are definable without parameters. Moreover, we have that $[a_1,\ldots,a_{n-1}]\in X \subseteq K$ and consequently
\[
(a_1,\ldots,a_{n-1})\in w_{G,n-1}^{-1}(K).
\] 
Thus, the set $w_{G,n-1}^{-1}(K)$ is $\mu_{G^{n-1}}$-wide, as so is $\tp(a_1,\ldots,a_{n-1})$. Hence, the inductive hypothesis yields that $G/K$ is (nilpotent of class less than $n-1$)-by-(finite exponent).

Let $F$ be a definable normal subgroup of $G$ containing $K$ such that $F/K$ is nilpotent of class less than $n-1$ and $G/F$ has finite exponent. Moreover, the subgroup $F$ is definable over the same parameters as $N$. We see now that $G/(F\cap H)$ has finite exponent and that $(F\cap H)/N$ is nilpotent of class less than $n$, which yields the result since $F\cap H$ is normal in $G$ and definable over the same parameters as $N$. To show this, observe that $C_{G}(x/N)$ has index at most $k$ in $G$ for every $x$ in $X$ and so 
$$
a^{k!} \in \bigcap_{x\in X} C_G(x/N) = C_G(X/N) = H
$$
for every element $a$ of $G$. Thus $G/H$ has exponent at most $k!$, which yields that $G/(F\cap H)$ has finite exponent. To prove that $(F\cap H)/N$ is nilpotent of class less than $n$, note that the $(n-1)$-th subgroup $\gamma_{n-1}(F)$ of the lower central series of $F$ is contained in $K=C_G(H/N)$. Thus $[\gamma_{n-1}(F),H] \le N$ and so
$$
\gamma_{n}(F\cap H)=[\gamma_{n-1}(F\cap H), F\cap H] \le [\gamma_{n-1}(F),H] \le N,
$$ 
yielding that $(F\cap H)/N$ is nilpotent of class less than $n$, as desired.
\end{proof}
 
 \begin{remark}
 	The proof of Theorem \ref{T:M2} yields that the subgroup $H$ is definable in the pure language of groups, when so is $N$. Indeed, we define $H=C_G(X/N)$ where $X$ is the set elements $x\in G$ such that $|G:C_G(x/N)|\le k$ for some fixed $k\ge 1$. Hence, we  obtain the following:
 	\begin{enumerate}
 		\item Every group automorphism leaving $N$ invariant leaves $H$ invariant.
 		\item If $\{N_i\}_{i\in I}$ is a uniformly definable family of normal subgroups satisfying the hypothesis of Theorem \ref{T:M2}, then the family $\{H_i\}_{i\in I}$ is uniformly definable too. In particular, by
 		compactness, the exponent of $G/H_i$ is uniformly bounded.
 	\end{enumerate} 
 \end{remark}
 
 As an immediate consequence we obtain the following version of Shalev's result.
\begin{cor}\label{C:Main}
 Let $G$ be a definable group. Suppose that $G$ admits a compatible system of ideals with $\mu_G$ satisfying the right translational $\mathrm{S}1$ property. Assume that the set 
 $w_{G,n}^{-1}(1)$
 is $\mu_{G^{n}}$-wide. Then, the group $G$ has an $\emptyset$-definable finite index subgroup which is FC-nilpotent of class $n-1$ and there is an $\emptyset$-definable characteristic subgroup $H$ of $G$ which is nilpotent of class less than $n$ and $G/H$ is uniformly locally finite.\end{cor}
\begin{proof}
By Proposition \ref{P:M1}, there exists an $\emptyset$-definable normal subgroup of $G$ which is FC-nilpotent of class $n-1$. In particular, the group $G$ is FC-nilpotent of class $n$. On the other hand, applying Theorem \ref{T:M2} we also obtain an $\emptyset$-definable normal subgroup $H$ of $G$ such that $H$ is nilpotent of class less than $n$ and $G/H$ has finite exponent; note that $H$ is characteristic by the previous remark. Since $G$ and so $G/H$ are $\mathrm{FC}$-nilpotent, we have that the group $G/H$ is locally finite by \cite[14.5.8.]{Rob}. Hence, it is uniformly locally finite by a standard compactness argument, since the ambient structure is $\aleph_1$-saturated.  
\end{proof}

\section{Applications to concrete families}\label{s:Applications}

In this section we apply the previous results to concrete families of groups. 

\subsection{Finite groups}\label{s:FiniteGroups} 
Fix a non-principal ultrafilter  $\mathcal U$ on $\mathbb N$, that is, a non-trivial finitely additive $\{0,1\}$-valued probability measure on subsets of $\mathbb N$ with the property that all finite subsets on $\mathbb N$ have measure $0$. Given an infinite sequence $(G_n)_{n\in\mathbb N}$ of finite groups with $\lim_{n\to+\infty}|G_n|=+\infty$, denote by $\mu_n$ the normalized counting measure on $G_n$. Consider the ultraproduct $G=\prod_{\mathcal U} G_n$ and set $\mu^{\times k}$ to be the ultralimit $\lim_{\mathcal U} \mu_n^{\times k}$ in $G^{k}$. 

The language $\mathcal L$ is the pure language of groups and the ambient $\aleph_1$-saturated structure is $\mathcal M=(G,\cdot)$. It is easy to verify that each $\mu^{\times k}$  is a left and right translation-invariant finitely additive probability measure on $\D_{G^{k}}(\mathcal M)$. Let $\lambda_k$ be the ideal of definable subsets of $G^k$ of $\mu^{\times k}$-measure $0$. It then follows that the ideal $\lambda_1$ has the translational S$1$ property \cite[Example 2.12]{Hr1} and the family $\{\lambda_k\}_{k\in\mathbb N}$ satisfies the Fubini property on ideals. The latter can be easily seen using Fubini theorem in the finite setting together with \L o\'s's theorem (cf. \cite[Section 2]{BT14}). In particular, an ultraproduct of finite groups admits a compatible system of ideals by Lemma \ref{L:Relation}.

Hence, as an easy consequence of Corollary \ref{C:Main} we obtain the following, which corresponds to Theorem \ref{T:IntroFinite} in the Introduction:

\begin{theorem}\label{T:Finite}
Given $\epsilon>0$ and natural numbers $d,k\ge 1$, there exists some $r=r(\epsilon,d,k)$ such that given a finite group $G$  with $\mathrm{Pr}_{G,w_k}(1)\ge \epsilon$, there exists some characteristic subgroup $N$ of $G$ which is nilpotent of class less than $k$ and such that every $d$-generator subgroup of $G/N$ is of order at most $r$.
\end{theorem}
\begin{proof}Suppose not. Thus, we can find some $\epsilon>0$ and some natural numbers $d,k\ge 1$ such that for every $n$, there exists a finite group $G_n$ with $\mathrm{Pr}_{G_n,w_k}(1)\ge \epsilon$ but $G_n/N$ has a $d$-generator subgroup of order at least $n$, for every characteristic subgroup $N$ of $G_n$ which is nilpotent of class less than $k$.
	
Let $G$ be the ultraproduct $\prod_{\mathcal U}G_n$ with respect to some non-principal ultrafilter $\mathcal U$. By the remarks above, the group $G$ satisfies the assumptions of Corollary \ref{C:Main} and thus, there exists an $\emptyset$-definable nilpotent subgroup $H$ of $G$, of nilpotency class less than $k$, such that $G/H$ is uniformly locally finite. In particular, every $d$-generator subgroup of $G/H$ is finite.

Set $H_n$ to be the trace in $G_n$ of the formula without parameters, in the pure language of groups, defining the subgroup $H$. By \L o\'s's theorem, since being nilpotent of class less than $k$ is first-order expressible, there are infinitely many natural numbers $n$ where $H_n$ is a characteristic subgroup of $G_n$ which is nilpotent of class less than $k$. Thus, by the choice of $G_n$ there are elements $g_{n,1},\ldots,g_{n,d}$ in $G_n$ such that the group $\langle g_{n,1}H_n,\ldots,g_{n,d} H_n \rangle / H_n$ has size at least $n$. However, if we let $g_i$ to denote the element of $G$ with the sequence $(g_{n,i})_{n\in\mathbb N}$ as a representative, then we obtain that $\langle g_{1}H,\ldots,g_{d} H \rangle / H$ is infinite, a contradiction. 
\end{proof}

As an immediate consequence we obtain in the finite setting a non-effective version of \cite[Theorem 1.1]{Sha} and of \cite[Corollary 1.4]{Sha}.

\subsection{Amenable groups}  A group $G$ is called {\em amenable} if it has a finitely additive right-transitive probability measure on $\mathcal P(G)$. In fact, this is equivalent to the existence of a right-invariant {\em mean} on $G$. We remind here that a mean is a linear map $\int \,\mathrm{d}\mu : \ell_\infty(G) \to \mathbb R$ such that $\int \mathds 1_{G} \,\mathrm{d}\mu = 1$ and it is positive in the sense that $\int f\, \mathrm{d}\mu \ge 0$ if $f\ge 0$ pointwise, for every $f\in \ell_\infty(G)$. The mean $\int \mathrm{d}\mu$ is right-invariant if $\int f\cdot g\, \mathrm{d}\mu = \int f\, \mathrm{d}\mu$ for every $g\in G$ and every $f\in \ell_\infty(G)$, where $f\cdot g(x) = f(xg^{-1})$ for every $x\in G$. 

Given two amenable groups $G$ and $H$, one can define a right-invariant mean on $G\times H$ via
\[
\int_{x\in G} \left( \int_{y\in H} f(x,y) \, \mathrm{d}\mu_G\right)  \mathrm{d}\mu_H,
\]
making $G\times H$ an amenable group. In particular, the following expression yields a finitely additive right-invariant probability measure on subsets of $G\times H$: 
\[
\mu_G\rtimes \mu_H(A) = \int_{x\in G} \left( \int_{y\in H} \mathds 1_A \, \mathrm{d}\mu_G\right)  \mathrm{d}\mu_H.
\]
We denote by $\mu^{\rtimes 2}$ the product measure  $\mu\rtimes \mu$ on $G^2$ and by $\mu^{\rtimes n+1}$ the product measure $\mu \rtimes \mu^{\rtimes n}$ on $G^{n+1}$.

Let $G$ be an amenable group and consider the product measures $\mu^{\rtimes n}$ for $n$ in $\mathbb N$. One can restrict each measure to the boolean algebra of definable, in the language of groups  $\mathcal L_{\mathrm{gr}}$, subsets of the corresponding $G^n$. Abusing notation, we denote such restriction by $\mu^{\rtimes n}$ as well. As explained in Example \ref{E:Meas}, we can expand the language of groups $\mathcal L_{\mathrm{gr}}$ to a countable language $\mathcal L^\mu$ in such a way that each measure $\mu^{\rtimes n}$ becomes definable: indeed, for each $n$, each $r\in \mathbb Q\cap [0,1]$ and each formula $\phi(x_1,\ldots,x_n, \bar y)$ without parameters in the language of groups, we put a predicate $Q_{n,r,\phi}(\bar y)$ such that 
\[
Q_{n,r,\phi}(\bar b)  \text{ holds } \ \Leftrightarrow \ \mu^{\rtimes n}(\phi(G^n,\bar b))\ge r.
\]
The predicates $Q_{n,r,\phi}(\bar y)$ give rise to new definable sets, which will also be measurable, since the corresponding Cartesian product of $G$ is amenable. Hence we can repeat the process adding new predicates $Q_{n,r,\phi}$ for these new formulas. Iterating this process we obtain the desired countable language $\mathcal L^\mu$.

As a consequence, seeing an amenable group $G$ as an $\mathcal L^\mu$-structure, the value of the measure of an $\mathcal L^\mu(G)$-formula does not depend on the model of the theory of $G$ that we are working in.

\begin{remark} Let $G$ be an amenable group seen as an $\mathcal L^\mu$-structure and let $\Gamma$ be an elementary extension of $G$. For each $i<n$, each $r\in\mathbb Q\cap [0,1]$ and each formula $\phi(x_i,x_{n-i};\bar y)$ in  $\mathcal L^\mu$, where $|x_i|=i$ and $|x_{n-i}|=n-i$, we have that if
$\mu^{\rtimes n}(\phi(\Gamma^n,\bar b))\ge r$, then the $\mathcal L^\mu$-definable set
\[
X_{i,r,\phi}=\left\{x \in \Gamma^i : \mu^{\rtimes n-i}(\phi(x,\Gamma^{n-i},\bar b))\ge \frac{r}{2} \right\} 
\]
has measure at least $r/2$.
\end{remark}
\begin{proof}
	Consider the definable set $Z =\{ a\in G^i : \mu^{\rtimes n-i}(\phi(a,G^{n-i},\bar b))\ge r/2\}$ and note that it is definable by the $\mathcal L^\mu$-formula $\psi(x_i,\bar b)$ given by $Q_{n-i,\frac{r}{2},\phi'}(x_i,\bar b)$, where $\phi'(x_{n-i};x_i,y)$ is the formula $\phi(x_i,x_{n-i},\bar y)$ with a distinct separation of variables. 
	
	An easy computation yields that
\begin{align*}
\mu^{\rtimes n}(\phi(G^n,\bar b)) & = \int_{a\in Z} \mu^{\rtimes n-i}(\phi(a,G^{n-i},\bar b)) \, \mathrm{d}\mu^{\rtimes i} +  \int_{a\not\in Z} \mu^{\rtimes n-i}(\phi(a,G^{n-i},\bar b)) \, \mathrm{d}\mu^{\rtimes i} \\
& \le \int_{a\in Z} 1\, \mathrm{d}\mu^{\rtimes i} + \int_{a\not\in Z}  \frac{r}{2} \, \mathrm{d}\mu^{\rtimes i} \le \mu(Z) + \frac{r}{2}
\end{align*}
and so $\mu(Z)\ge r/2$ if $\mu^{\rtimes n}(\phi(G^n,\bar b))\ge r$. This implies that in $G$, for any tuple $\bar b$ we have that  $Q_{1,\frac{r}{2},\psi}(\bar b)$ holds whenever so does $Q_{n,r,\phi}(\bar b)$. Hence, the same is true for any tuple of $\Gamma$ since $\Gamma$ is an elementary extension of $G$, which yields the statement.
\end{proof}	
\begin{cor}\label{C:DefAmCompatible}
	Let $G$ be an infinite amenable group and let $\Gamma$ be an $\aleph_1$-saturated elementary extension of $G$, as $\mathcal L^\mu$-structures. The family of definable measures $\{\mu^{\rtimes n}\}_{1\le n}$ induces a compatible system of ideals on the boolean algebra of $\mathcal L^\mu$-definable subsets of $\Gamma$.
\end{cor}
\begin{proof}
It suffices to show that $\mu^{\rtimes n}$ is coordinatewise $(\mu^{\rtimes i}, \mu^{\rtimes n-i})$-compatible for any $1\le i<n$. To see this, let $Y\subset \Gamma^n$ be $A$-definable, defined by an $\LL_\mu(A)$-formula $\phi(x_i,x_{n-i})$, with $\mu^{\rtimes n}(Y)\ge r>0$, for some $r\in (0,1]\cap \mathbb Q$. By compactness and $\aleph_1$-saturation, the previous remark implies that there exists some $a\in X_{i,r,\phi}$ such that $\tp(a/A)$ is $\mu^{\rtimes i}$-wide and $\mu^{\rtimes n-i}(Y_{a})\ge r/2$, where $Y_a$ is the set defined by $\phi(a,x_{n-i})$. As $\mu^{\rtimes n-i}(Y_a)\ge r/2$, again by compactness and $\aleph_1$-saturation we obtain some $b$ in $\Gamma^{n-i}$ whose type over $A\cup\{a\}$ is $\mu^{\rtimes n-i}$-wide and $(a,b)\in Y$. Hence, we obtain the result.
\end{proof}

As an immediate consequence of Theorem \ref{T:M2}, we obtain Theorem \ref{T:IntroAmen} from the introduction:

\begin{theorem}\label{T:Amen}
	Let $G$ be an infinite amenable group with a finitely additive right-invariant probability measure $\mu$ and suppose that
	\[
	\mu^{\rtimes k}(w_{G,k}^{-1}(1)) = \int_{x_1\in G} \left( \dots \left( \int_{x_k\in G} \mathds{1}_{w_{G,k}^{-1}(1)} (x_1,\ldots,x_k) \,\mathrm{d} \mu \right) \ldots \right) \,\mathrm{d} \mu >0.
	\]
	We have that $G$ has a finite index $\mathcal L_{\mathrm{gr}}$-definable subgroup which is FC-nilpotent of class at most $k-1$ and there exists some $\mathcal L_{\mathrm{gr}}$-definable subgroup $H$ of $G$ which is nilpotent of class less than $k$ and such that $G/H$ is uniformly locally finite.
\end{theorem}
\begin{proof} We regard $G$ as an $\LL^\mu$-structure and consider an $\aleph_1$-saturated elementary extension $\Gamma$ of $G$. By Corollary \ref{C:DefAmCompatible}, we have that $\Gamma$ admits a compatible system of ideals with the right-translational S$1$ property, which is induced by the family $\{\mu^{\rtimes n}\}_{n\in \mathbb N}$ of measures. Also, note that $\mu^{\rtimes  k}(w_{\Gamma,k}^{-1}(1))>0$ by definability of $\mu^{\rtimes k}$ and so $\Gamma$ satisfies the assumptions of Corollary \ref{C:Main}.  Thus, the group $\Gamma$ and so $G$ have a finite index subgroup which is FC-nilpotent of class at most $k-1$. Furthermore, there is a formula $\vartheta(x)$ without parameters (in the language of groups) such that $\vartheta(\Gamma)$ is a subgroup of $\Gamma$ which is nilpotent of class less than $k$ and such that $\Gamma/\vartheta(\Gamma)$ is uniformly locally finite. Therefore, we obtain that $\vartheta(G)=G\cap \vartheta(\Gamma)$ is an $\emptyset$-definable (in the language of groups) subgroup of $G$ satisfying the statement, since it is a subgroup of $\vartheta(\Gamma)$ and $G/\vartheta(G)$ embeds into $\Gamma/\vartheta(\Gamma)$.
\end{proof}

As in the finite case, a standard compactness argument yields that the function witnessing that the $d$-generator subgroups of $G/H$ have bounded size only depend on $n$, $d$ and the value of $\mu^{\rtimes n}(w_{G,n}^{-1}(1))$. We leave the details to the reader. Furthermore, as a consequence, by \cite[Corollary 2.1]{DM56} we obtain the following particular case of \cite[Theorem 1.6]{MTVV21}.
\begin{cor}
Let $G$ be an infinite finitely generated amenable group with a finitely additive right-invariant probability measure $\mu$ and suppose that
	\[
	\mu^{\rtimes k}(w_{G,k}^{-1}(1)) = \int_{x_1\in G} \left( \dots \left( \int_{x_k\in G} \mathds{1}_{w_{G,k}^{-1}(1)} (x_1,\ldots,x_k) \,\mathrm{d} \mu \right) \ldots \right) \,\mathrm{d} \mu >0.
	\]
	We have that $G$ is the extension of a nilpotent of class at most $k-1$ group by a finite group.
\end{cor}

For $n=2$,  Corollary \ref{C:1} yields Neumann's statement for amenable groups, by mimicking the proof above (see \cite[Theorem 1.14]{Toi} for an effective quantitative version).

\begin{theorem}
	Let $G$ be an amenable group with a finitely additive right-invariant probability measure $\mu$ and suppose that $\mu^{\rtimes 2}(w_{G,2}^{-1}(1))>0$. We have that $G$ is finite-by-abelian-by-finite.
\end{theorem}

\subsection{Forking generics}
Given a definable group $G$, we say that a formula $\phi(x)$ is {\em $f$-generic} over a set $A$  if for every $g$ in $G$ the formula $g\cdot \phi(x)$ does not fork over $A$, where $g\cdot \phi(x)$ denotes the formula $\exists z \left( \phi(z) \wedge gz=x \right)$; see Example \ref{E:Fork}. We say that a group $G$ has strong $f$-generics if there exists a small model $M_0$ and a global type $p(x)$ all whose formulas are $f$-generic over $M_0$.

Under the assumption that the ambient theory is simple, the collection of all formulas which are not $f$-generic over $\emptyset$ forms an ideal and every definable group in a simple theory has strong $f$-generics. This can be easily seen by combining Proposition 4.1.7 and Lemma 4.1.9 in \cite{Wag}. Furthermore, the general theory of forking in simple theories yields that this ideal has the translational S$1$ property (see \cite[Theorem 2.5.4]{Wag}) and the family of non-$f$-generics in the respective Cartesian products satisfies Fubini. To see this, note that a formula $\phi(x)\wedge \psi(y)$ is $f$-generic if and only if $\phi(x)$ and $\psi(y)$ are. Moreover, a type $\tp(a,b)$ is $f$-generic in $G\times H$ if and only if the types $\tp(a)$ and $\tp(b)$ are $f$-generic and $a$ is independent from $b$, see for instance \cite[Lemma 4.1.2]{Wag} as well as \cite[Lemma 4.3.12]{Wag}. 

Therefore, applying the results from the previous section we obtain the following:

\begin{proposition}
 Let $G$ be a definable group in a simple theory and assume that the set $w_{G,2}^{-1}(1)$ is $f$-generic. Then the group $G$ is finite-by-abelian-by-finite. 
\end{proposition}

\begin{theorem}
 Let $G$ be a definable group in a simple theory and assume that the set $w_{G,n}^{-1}(1)$ is $f$-generic. Then there exists a definable subgroup $H$ of $G$ which is nilpotent of class less than $2n-2$ and such that $G/H$ is finite.
\end{theorem}
\begin{proof} 
 Without loss, we may assume that $G$ is $\aleph_1$-saturated. By Proposition \ref{P:M1} the group $G$ has a finite index subgroup which is FC-nilpotent of class $n-1$ and therefore by \cite[Lemma 3.4]{PaWa} we obtain that $G$ has a definable finite index subgroup of nilpotency class less than $2(n-1)$.  
\end{proof}


\subsection{Fsg groups} A definable group $G$ has {\em finitely satisfiable generics}  (fsg, in short) if there exists a small model $M_0$ and a global type $p(x)$ such that any translate $g\cdot p(x)$ of $p(x)$ is finitely satisfiable in $M_0$. In particular, a group with fsg has strong $f$-generics. Groups with fsg were first studied in \cite{HPP}. 

A definable set $X$ is said to be {\em left generic} in $G$ if finitely many left translates of $X$ cover $G$. Likewise with right generic. The set $X$ is {\em generic}
if it is both left and right generic. A partial type $\pi(x)$ implying $G(x)$ is generic if every formula in $\pi(x)$ is. In a group having fsg, witnessed by $p(x)$ and $M_0$, the set of non-generic sets forms an ideal and thus generic types exist. In fact, a subset $X$ is generic if and only if every translate of $X$ is satisfied in $M_0$. Hence, the type $p(x)$ is generic, as is any translate of $p(x)$ (cf.\,\cite[Proposition 4.2]{HPP}). In addition, there is a bounded number of (global) generic types \cite[Corollary 4.3]{HPP} and consequently, an easy compactness argument yields:

\begin{fact}
The collection of non-generic sets forms an ideal which has the right (left) translation S$1$ property.
\end{fact}

Furthermore, \cite[Proposition 4.5]{HPP} yields that a finite Cartesian product of groups with fsg has also fsg.  We give a shorter proof of this. Before proceeding, we point out that given two finitely satisfiable over $M$ global types $p$ of $G$ and $q$ of $H$, we can consider two global types in $G\times H$. Namely, consider
\[
p\rtimes q = \tp(a,b/\M), 
\]
where $(a,b)$ realizes $p\times q$ in some elementary extension of $\M$ and moreover $b$ realizes the unique finitely satisfiable extension of $q$ over $\M\cup\{a\}$. Likewise, one can define $p\ltimes q$ imposing that $a$ realizes the unique finitely satisfiable extension of $p$ over $\M\cup\{b\}$.

\begin{fact}
Let $G$ and $H$ be two definable groups with fsg witnessed by $p(x)$ and $q(y)$ respectively. Then $p\rtimes q$ and $p\ltimes q$ witness that $G\times H$ has fsg.
\end{fact}
\begin{proof}
Let $M_0$ be some model with the property that every translate of $p(x)$ and of $q(y)$ is finitely satisfiable over $M_0$. We prove that $(g,h)\cdot p\rtimes q$ is finitely satisfiable over $M_0$, for a given $(g,h)\in G\times H$. Let $M\preceq \mathcal M$ be some small model containing $M_0, g$ and  $h$. Let $a$ be a realization of $p_{|M}$ and let $b$ be a realization of $q_{|M,a}$. By the assumption on $p$ and $q$, we get that $\tp(g^{-1}a/M)$ and $\tp(h^{-1}b/M,a)$ are finitely satisfiable over $M_0$ and hence $\tp(g^{-1}a,h^{-1}b/M) = (g,h)\cdot \left(p\rtimes q\right)_{|M}$ is finitely satisfiable over $M_0$, by transitivity. This yields the result. Likewise for $p \ltimes q$. 
\end{proof}

Therefore, for a group $G$ with fsg over $M_0$, we obtain that the collection of non-generic sets is an $M_0$-invariant ideal and that this family of ideals is a compatible system:
\begin{lemma}
Let $G$ be a definable group with fsg. If $Y\subset G^k$ is definable over $A$ generic, then for every $i$ there exist some $(a,b)\in G^i\times G^{n-i}$ such that $(a,b)\in Y$ and the types $\tp(a/A)$ and $\tp(b/A,a)$ are generic.
\end{lemma}
\begin{proof}
Let $p$ be a global type witnessing that $G$ has fsg over a model $M_0$. Suppose that $Y\subset G^k$ is $A$-definable and generic. Let $M$ be a model containing $M_0$ and $A$. Note that there is some finite subset $C$ of $G^k\cap M^k$ such that $G^n = C \cdot Y$. Let $(a,b)$ be a realization of $\left(p^{\rtimes i}\rtimes p^{\rtimes k-i }\right)_{|M}=(p^{\rtimes k})_{|M}$. Since $p^{\rtimes k}$ witnesses that $G^k$ has fsg, the type $(p^{\rtimes k})_{|M}$ is generic and thus, it contains some translate of the set $Y$ given by a tuple $(g,h)\in G^k\cap M^k$. Therefore $(g^{-1}a,h^{-1}b)\in Y$  and  the type $\tp(g^{-1} a/M)$, as well as the type $\tp(h^{-1}b/M,a)$ are generic, since the tuple $(g^{-1}a,h^{-1}b)$ realizes $(g,h) \cdot \left(p^{\rtimes k}\right)_{|M} = \left((g,h) \cdot  p^{\rtimes k}\right)_{|M}$. Hence, the types $\tp(g^{-1}a/A)$ and $\tp(h^{-1}b/A,a)$ are generic, as desired.
\end{proof}

Altogether, Corollary \ref{C:Main} can be applied to definable groups with fsg.  In fact, in this situation one can obtain a stronger result. The reason is that given a definable group $G$ with fsg there exists a smallest type-definable subgroup of $G$ of index at most $2^{\aleph_0}$, which is denoted by $G^{00}$ (cf.\,\cite[Corollary 4.3]{HPP}). Since $G^{00}$ exists, there also exists a smallest type-definable subgroup of bounded index which is the intersection of definable (normal) subgroups. It is denoted by $G^0$.

\begin{theorem}
Let $G$ be a definable group with fsg such that the set $w_{G,k}^{-1}(1)$ is generic in $G^k$. Then there exists a definable normal subgroup $N$ of $G$ of finite index which is nilpotent of class at most $k-1$. 
\end{theorem}
\begin{proof}
Observe first that since $G^0$ is type-definable, the iterated centralizer subgroups of $G^0$, which are defined inductively on $i$ as
\[
C_G^{i+1}\left(G^0\right) = C_G\left(G^0/C_G^i(G^0)\right),
\]
are $\emptyset$-ind-definable by Remark \ref{R:Ind-def}. By the previous lemma, there are $a_1,\ldots, a_k$ in $G$ such that $[a_1,\ldots,a_k]=1$ and that $\tp(a_1)$ and $\tp(a_{i+1}/a_1,\ldots,a_{i})$ for $i=1,\ldots,k-1$ are generic. We prove by induction on $i\le k-1$ that 
\[
[a_1,\ldots,a_{k-i}]\in C_G^i(G^0). 
\] 
For $i=0$, there is nothing to prove since $C_{G}^0(G^0)=1$. 
Suppose that the statement holds for $i<k-1$. Thus 
\[
a_{k-i} \in  C_G \left([a_1,\ldots,a_{k-(i+1)}]/C_{G}^i(G^0) \right)
\]
and so the subgroup $C_G \left( [a_1,\ldots,a_{k-(i+1)}]/C_{G}^i(G^0) \right)$ has finite index in $G$ and it is definable by Lemma \ref{L:WideInd}, since it is ind-definable over $a_1,\ldots,a_{k-(i+1)}$ by Remark \ref{R:Ind-def}. In particular, it contains $G^0$ and so 
\[
[a_1,\ldots,a_{k-(i+1)}]\in C_G \left( G^0/C_G^i(G^0) \right) = C_G^{i+1}(G^0),
\]
as desired. Hence, we finally obtain that $a_1$ belongs to $C_G^{k-1}(G^0)$ and so $C_G^{k-1}(G^0)$ is definable and has finite index, by Lemma \ref{L:WideInd}. Thus, we have that $G^0\le C_G^{k-1}(G^0)$ and so $G^0$ is nilpotent of class at most $k-1$. Therefore, since $G^0$ is the intersection of all definable finite index normal subgroups of $G$, a standard compactness argument yields the existence of a definable normal subgroup $N$ of $G$ which is nilpotent of class at most $k-1$ and has finite index in $G$.
\end{proof}

This result applies to definably compact groups definable in a saturated o-minimal expansion of a real closed field (cf.\,\cite[Theorem 8.1]{HPP}). Consequently, for compact semialgebraic Lie groups we obtain the following (see \cite[Example 8.34]{Sim}).

\begin{cor}\label{C:Lie}
Let $k\ge 1$ and let $G$ be a compact semialgebraic Lie group and suppose that $w_{G,k}^{-1}(1)$ has positive Haar measure. Then $G$ has a clopen finite index subgroup which is nilpotent of class at most $k-1$.
\end{cor}
	
We remark that under certain assumptions on the ambient theory the ideal of non-generic sets satisfies the Fubini property. This holds for definable groups with fsg in an NIP theory, see \cite[Proposition 8.32]{Sim} and \cite[Theorem 3.29]{Star}.

\bibliographystyle{plain}

\end{document}